\documentclass[12pt,reqno]{amsart}

\usepackage{amssymb, amsmath, amsfonts, latexsym}
\usepackage{cases}
\usepackage{bbm}
\usepackage{mathabx}
\usepackage{mathrsfs}
\usepackage{graphicx}
\usepackage{enumerate}

\newtheorem{thm}{Theorem}[section]

\newtheorem{lemma}[thm]{Lemma}
\newtheorem{proposition}[thm]{Proposition}

\newtheorem{remarks}[thm]{Remark}

\theoremstyle{definition}
\newtheorem{defn}{Definition}[section]
\theoremstyle{remark}

\newcommand{\ee}{\mathbb{E}}

\newcommand{\nn}{\mathbb{N}}

\newcommand{\pp}{\mathbb{P}}

\newcommand{\GG}{\mathcal{G}}

\def\AA{\mathcal A}

\def\DD{\mathcal D}
\def\FF{\mathcal F}
\def\EE{\mathcal E}

\def\HH{\mathcal H}

\def\beq{\begin{equation}}
\def\nneq{\end{equation}}

\def\bdef{\begin{defn}}
\def\ndef{\end{defn}}

\def\bthm{\begin{thm}}
\def\nthm{\end{thm}}

\def\bprop{\begin{prop}}
\def\nprop{\end{prop}}

\def\brmk{\begin{remarks}}
\def\nrmk{\end{remarks}}

\def\bexa{\begin{exa}}
\def\nexa{\end{exa}}

\def\blem{\begin{lem}}
\def\nlem{\end{lem}}

\def\bcor{\begin{cor}}
\def\ncor{\end{cor}}

\def\bexe{\begin{exe}}
\def\nexe{\end{exe}}

\def\bprf{\begin{proof}}
\def\nprf{\end{proof}}

\def\bdes{\begin{description}}
\def\ndes{\end{description}}

\def\e{\varepsilon}
\def\<{\langle}
\def\>{\rangle}

\title[CLT and MDP for a stochastic Cahn-Hilliard equation ]{Central limit theorem and moderate deviations for a stochastic Cahn-Hilliard equation}

\author{Ruinan Li}
\address{Ruinan Li \\
School of Statistics and Information, Shanghai University of International Business and Economics, Shanghai, 201620,  PR China.}
\email{ruinanli@amss.ac.cn}

\author{Xinyu Wang}
\address{Xinyu Wang \\
School of Mathematics and Statistics, Huazhong University of Science and Technology,  Wuhan,  430074, China.}

\email{wang\_xin\_yu@hust.edu.cn}

\date{}

\begin{document}
\maketitle

 \noindent {\bf Abstract:}
 In this paper, we prove a central limit theorem and   a moderate deviation principle for a perturbed  stochastic Cahn-Hilliard equation
 defined on $[0,T]\times [0,\pi]^d$ with $d\in\{1, 2, 3\}$. This equation is driven by a space-time white noise.  The weak convergence approach  plays an important role.
 \vskip0.3cm

 \vskip0.3cm

\vskip0.3cm

\section{Introduction}

\noindent Since the pioneer works of Freidlin and Wentzell \cite{FW}, the theory of small perturbation large deviation  principle (LDP for short) for stochastic (partial) differential equations has been extensively developed,  see  monographs \cite{DPZ,DZ,DE}, and  papers  \cite{BDM, CM,  CR,  RZ, XZ} and references therein for further developments.  Like the large deviations, the moderate deviation  problems arise in the theory of statistical  inference quite naturally. The  moderate deviation  principle (MDP for short) can provide us with the rate  of convergence and a useful method for constructing asymptotic confidence intervals, see \cite{ DGS,Erm, GZ}  and references therein.

 The problem of moderate deviations for stochastic partial differential equations  has been receiving much attention in very recently years, such as Wang and Zhang \cite{WZ} for the stochastic reaction-diffusion equations, Wang {\it et al.} \cite{WZZ} for the stochastic Navier-Stokes equations, Li {\it et al.} \cite{LWYZ} for the fractional stochastic heat equation, Yang and Jiang \cite{YJ} for the fourth-order stochastic heat equations with fractional noises, Gao and Wang \cite{GW} for   the forward-backward stochastic differential equations. By using a new weak convergence approach,  Budhiraja {\it et al.} \cite{BDG} and Dong {\it et al.} \cite{DXZZ} studied the MDP for stochastic dynamics  with jumps.

  Since the work of Da Prato  and  Debussche  \cite{DPD}, it has  attracted much attention in the study of stochastic  Cahn-Hilliard equation, which  describes phase separation in a binary alloy  in the presence of thermal fluctuations (see \cite{BMW}). For example,    see  Cardon-Weber \cite{CW} for  the stochastic Cahn-Hilliard equation driven by the Guassian space-time white noise, Bo and Wang \cite{BW} for    the  L\'evy  space-time white noise case, Bo {\it et al.} \cite{BJW} for the fractional noises case, and so on.  The ergodic properties and invariant
measures of the stochastic  Cahn-Hilliard equation with degenerate noise were studied in Gouden\`ege  and Manca \cite{GM}. The sharp interface limits  for the stochastic  Cahn-Hilliard equation are extensively studied,  see Funaki \cite{F} and references therein for further developments.

\vskip 0.3cm
Wentzell-Freidlin's large deviation results for stochastic  Cahn-Hilliard equation  have been established in   \cite{STW} and \cite{BM}. In this paper, we shall study the     central limit theorem and  MDP for this equation.

\vskip 0.3cm

The rest of this paper is organized as follows. In Section 2, we give the framework of the stochastic Cahn-Hilliard equation. In Section 3, we prove the central limit theorem. Section 4 is devoted to the proof of the  MDP by using the weak convergence method.

\vskip0.3cm

Throughout the paper,  the generic positive constant $C$ may change from line to line. If it is essential, the dependence of a constant $C$ on some parameters, say $T$, will be written by $C(T)$. We denote the space variable by $x$ and the space integral by $\int \cdot  dx$,  and for any $p\ge 1$ denote $\|\cdot\|_p$ the $L^p$-norm with respect to $dx$.

\vskip0.3cm
Let  $\DD=[0,\pi]^d$ with $d\in\{1,2,3\}$. For any $T>0, p\geq 1,\alpha\in(0,1)$,  let $C^{\alpha}([0,T], L^p(\DD))$  be the H\"{o}lder space equipped with the norm defined by
$$
\|f\|_{\alpha, p}:=\sup_{t\in[0,T]}\|f(t,\cdot)\|_p +\sup_{s\neq t,s,t\in[0,T]}\frac{\|f(t,\cdot)-f(s,\cdot)\|_p}{|t-s|^\alpha},
$$
 and
$$
C^{\alpha,0}([0,T], L^p(\DD)):=\left\{ f\in C^{\alpha}([0,T], L^p(\DD)); \lim_{\delta\rightarrow0} O_{f}(\delta)=0\right\},
$$
where $O_{f}(\delta):=\sup_{|t-s|<\delta}\frac{\|f(t,\cdot)-f(s,\cdot)\|_p}{|t-s|^{\alpha}}$.  Then  $C^{\alpha,0}([0,T],L^p(\DD))$ is a Polish space,  which is denoted  by $\EE_\alpha$.

\section{Framework and main results}

\subsection{Framework}

For any $T>0,\e>0$, consider
\begin{equation}\label{SPDE}
    \begin{cases}
     \frac{\partial u^{\e}}{\partial t}(t,x)=-\Delta\left(\Delta u^{\e}(t,x)-f(u^{\e}(t,x))\right)+\sqrt{\e}\sigma(u^\e(t,x))\dot{W}(t,x),
     \text{  in }[0,T]\times\DD,\\
      u^\e(0,x)=u_{0}(x),\\
     \frac{\partial u^\e}{\partial {\bf n}}(t,x)=\frac{\partial \Delta u^\e}{\partial {\bf n}}(t,x)=0,  \text{  on }[0,t]\times\partial\DD,
    \end{cases}
\end{equation}
where   $\Delta u^{\e}$ denotes the Laplacian of $u^{\e}$ in the $x$-variable; $\dot{W}$ is a   Gaussian space-time white noise on some probability space $(\Omega, \FF,\{\FF_t\}_{t\geq 0}, \pp)$; ${\bf n}$ is the outward normal vector; the functions $\sigma, f$ and  $u_0$ satisfy the following {\bf Hypothesis (H)}:
\begin{quote}
\begin{itemize}
        \item[({\bf H.1})] $\sigma$ is a bounded and  Lipschitz  function;

        \item[({\bf H.2})] $f$ is a  polynomial of degree $3$ with positive dominant coefficient;

        \item[({\bf H.3})] $u_0$ is a continuous function on $\DD$ which belongs to $L^p(\DD)$ for $p\geq 4$;
        \item[({\bf H.4})] $u_0$ is a $\gamma$-H\"{o}lder continuous function on $\DD$ with $\gamma \in (0, 1)$.
\end{itemize}
\end{quote}

\vskip0.3cm

According to Cardon-Weber \cite{CW}, under hypothesis ({\bf H.1})-({\bf H.3}), Eq.\eqref{SPDE} admits a unique solution $u^\e$ in the following mild form:
\begin{align}\label{SPDE solution}
u^{\e}(t,x)=&\int_{\DD}G_t(x,y)u_0(y)dy +\int_0^t\int_{\DD}\Delta G_{t-s}(x,y)f(u^{\e}(s,y))dyds\notag\\
           &+ \sqrt{\e}\int_0^t\int_{\DD} G_{t-s}(x,y)\sigma(u^{\e}(s,y))W(ds,dy),
\end{align}
 where $G_t(\cdot, \cdot)$  is the Green function corresponding to the operator $\frac{\partial}{\partial t}+\Delta^2$ with Neumann boundary conditions,   which satisfies
\beq\label{eq u e estimate}
\sup_{\e\leq 1}\sup_{t\in[0,T]}\ee\left[\left\|u^{\e}(t,\cdot)\right\|_p^q\right]<+\infty,
\nneq
for $q\geq p$,
see \cite[Theorem 1.3]{CW}. Moreover, Therorem 1.4 in \cite{CW} asserts that under hypothesis ({\bf H}), the trajectories of the solution are a.s. $\alpha$-H\"{o}lder continuous in $t$ with $\alpha\leq \frac{\gamma}{4}, \alpha < \frac{1}{2}(1-\frac{d}{4})$. Consequently, the random field solution $\{u^\e(t,x); (t,x)\in[0, T]\times \DD\}$ to Eq.\eqref{SPDE} lives in space $\EE_\alpha$.

\vskip0.3cm
Intuitively, as the parameter $\e$ tends to zero, the solution $u^\e$ of $(\ref{SPDE})$ will tend to the solution of the deterministic  equation
\beq\label{eq u0}
u^0(t,x)=\int_\DD G_t(x,y)u_0(y)dy+ \int_0^t\int_{\DD}\Delta G_{t-s}(x,y)f(u^0(s,y))dyds.
\nneq
\vskip0.3cm
Particularly, taking $\e=0$ in Eq.$(\ref{SPDE solution})$,   the   solution of \eqref{eq u0}  has the following estimate
\beq\label{eq u 0 estimate}
\sup_{t\in[0,T]}\left\|u^0(t,\cdot)\right\|_p<+\infty.
\nneq
In this paper, we shall investigate deviations of $u^\e$ from  $u^0$, as $\e$ decreases to $0$. That is  the
asymptotic behavior of the trajectories,
\beq\label{U}
Z^\e(t,x):=\frac{1}{\sqrt{\e}h(\e)}(u^\e-u^0)(t,x),\quad(t,x)\in [0,T]\times \DD.
\nneq
\begin{itemize}
  \item[(LDP)]
 The case $h(\e)=1/\sqrt\e$ provides some large deviation  estimates.  Shi {\it et al.} \cite{STW},   Boulanba and  Mellouk  \cite{BM} proved that the law of the solution $u^{\e}$  satisfies an LDP.
  \item[(CLT)]
  If $h(\e)\equiv 1$, we are in the domain of the central limit theorem (CLT for short).
We will show that $(u^\e-u^0)/\sqrt\e$ converges  to a  random field  as $\e\rightarrow 0$,   see Theorem \ref{CLT} below.
  \item[(MDP)]
  To fill in the gap between the CLT scale and the large deviations scale,
we will study moderate deviations, that is when the deviation scale satisfies
\begin{equation} \label{h}
 h(\e)\to+\infty \ \ \text{and }\quad\sqrt\e h(\e)\to0,\quad \text{as}\quad\e\to0.
 \end{equation}
 In this case, we will prove that $Z^{\e}$ satisfies an LDP,  see Theorem \ref{MDP} below.  This special type of LDP is called the MDP for $u^\e$, see \cite[Section 3.7]{DZ}.

\end{itemize}

Throughout this paper, we assume that \eqref{h} is in place.

\subsection{Main results}

Let
\begin{equation}\label{eq Y}
    \begin{cases}
     \frac{\partial Y}{\partial t}(t,x)=-\Delta\left(\Delta Y(t,x)-f'(u^0(t,x))Y(t,x)\right)+\sigma(u^0(t,x))\dot{W}(t,x),
     \text{  in }[0,T]\times\DD,\\
     Y(0,x)=0,\\
     \frac{\partial Y}{\partial {\bf n}}(t,x)=\frac{\partial \Delta Y}{\partial {\bf  n}}(t,x)=0,\quad   \text{  on }[0,T]\times \partial\DD.\\
     \end{cases}
\end{equation}

Following the proofs of  Theorems 1.3 and 1.4 in \cite{CW}, it is easy to obtain that
under ${\bf(H)}$,  Eq.\eqref{eq Y} admits a unique solution $Y$,  which satisfies
\beq\label{eq Y estimate}
\sup_{t\in[0,T]}\ee\left[\|Y(t,\cdot)\|_p^q\right]<+\infty,
\nneq
for $q\geq p$.
Furthermore, for any $t,t' \in[0,T]$,
\begin{equation}\label{holder-Y}
\ee\left[\left\|Y(t',\cdot)-Y(t,\cdot)\right\|_p^q\right] \leq C|t'-t|^{\alpha q}, \ \  \pp\text{-a.s.}
\end{equation}
with $\alpha\leq \frac{\gamma}{4}, \alpha < \frac{1}{2}(1-\frac{d}{4})$ .
  \vskip0.3cm

Our first main result is the following central limit theorem.
\bthm \label{CLT}
{\rm
Under   {\bf(H)}, for any $\alpha\leq \frac{\gamma}{4}$ and $\alpha < \frac{1}{2}(1-\frac{d}{4})$, the random field $(u^\e-u^0)/\sqrt{\e}$ converges in probability to a random field $Y$ defined by  \eqref{eq Y} on $\EE_\alpha$.
  }
\nthm

Noticing that the coefficients $f'(u^0)$ and $\sigma(u^0)$ are bounded and independent of $Y$, Eq.\eqref{eq Y} becomes a fourth-order stochastic heat equation, by using the classical Freidlin-Wentzell's large deviation theory  (e.g., \cite{JSW}),   it is easily to obtain that $Y/h(\e)$ obeys an LDP on $\EE_\alpha$ with the speed $h^2(\e)$ and with the good rate function:
\begin{equation}\label{rate function1}
 I(g):=\left\{
       \begin{array}{ll}
         \inf_v \left\{\frac12\int_0^T\int_{\DD}v^2(t,x)dxdt;  Z^v=g\right\},   & \hbox{\text{if} $g\in \textit{Im}(Z^{\cdot})$;}\\
        +\infty, & \hbox{\text{otherwise},}
       \end{array}
     \right.
\end{equation}
where the infimum is taken over all $v\in L^2([0,T]\times \mathcal D)$ and the functional $Z^v$ is the solution of the following  deterministic partial differential equation
\begin{align}\label{eq sk}
Z^{v}(t,x)=&\int_0^t\int_{\DD} G_{t-s}(x,y)\sigma(u^0(s,y))v(s,y)dyds\notag\\
            &+\int_0^t\int_{\DD} \Delta G_{t-s}(x,y)f'(u^0(s,y))Z^v(s,y)dyds.
\end{align}

\vskip0.3cm

Our second main result is that the law of $\{(u^\e-u^0)/[\sqrt{\e}h(\e)]; \e\in (0,1]\}$ obeys the same LDP with  $Y/h(\e)$, that is  the following theorem.
\bthm\label{MDP}{\rm Under  {\bf(H)},  the family $\{(u^\e-u^0)/[\sqrt{\e}h(\e)];\e\in (0,1]\}$ satisfies a
large deviation  principle on $\EE_\alpha$ with the speed function $h^2(\e) $ and with the good rate function $I$ given by \eqref{rate function1}.  More precisely,  the following  statements hold:
       \begin{itemize}
         \item[$(a)$] for each $M<\infty$, the level set $\{ f\in\EE_\alpha;I(f)\leq M\}$ is a compact subset of $\EE_\alpha$;
         \item[$(b)$] for each closed subset $F$ of $\EE_\alpha$,
              $$
                \limsup_{\e\rightarrow 0}\frac1{h^2(\e)}\log\mathbb{P}\left(\frac{u^\e-u^0}{\sqrt{\e}h(\e)}\in F\right)\leq- \inf_{f\in F}I(f);
              $$
         \item[$(c)$] for each open subset $G$ of $\EE_\alpha$,
              $$
                \liminf_{\e\rightarrow 0}\frac1{h^2(\e)}\log\mathbb{P}\left(\frac{u^\e-u^0}{\sqrt{\e}h(\e)}\in G\right)\geq- \inf_{f\in G}I(f).
              $$
       \end{itemize}
        }
\nthm

\brmk {\rm
\begin{itemize}
 \item [(1)] It is well-known that the   Freidlin-Wentzell's theory  is very powerful for   the  studies  of varies  asymptotics and their applications to problems of the behavior of a random process on large time intervals, such as the problem of the limit behavior of the invariant measure, the problem of exit of a random process from a domain, and the problem of stability under random perturbations, see \cite{FW}.  The reader is referred to the recent work by   Tao \cite{tao} for a lot of literature on the development of numerical methods.

  Usually,  the relatively simple form of the corresponding rate function in Theorem \ref{MDP} suggests
that many of the constructions needed to implement an effective importance sampling scheme
 would be simpler than in the corresponding large deviation context, e.g., Klebaner and Liptser \cite{Klebaner1999},  Dupuis and Wang \cite{DW}.  Refer to Budhiraja {\it {\it et al.}} \cite{BDG} for  further development and references.
 \item [(2)] By using the classical exponential approximation method,  Yang and Jiang \cite{YJ} proved a  MDP  for a fourth-order stochastic heat equation  with fractional noises, in which  the drift term $f$ together with its  derivative are assume  to be Lipschitz and the diffusion coefficient $\sigma=1$.    Due to the non-Lipschitz and nonlinearity of $\Delta f$ and the multiplicative noise  in  \eqref{SPDE},  it is difficult to get some exponential estimates and the classical exponential approximation method  would become rather complicated   for  stochastic  Cahn-Hilliard equation  \eqref{SPDE}.  Here, we use the weak convergence approach \cite{BDM, DE}, in which
       one only needs  some   moment estimates.
 \end{itemize}
}
\nrmk

\section{Central Limit Theorem}
\subsection{Preparation}

First, let us recall some estimations on the Green function $G$ which will be used quiet often later. These properties can be found in \cite{BJW,CW}.

\begin{lemma}\label{lemma1}
\begin{itemize}
\item[1.]
{\rm
There exists some positive constant $C$ such that for any $t\in(0,T]$ and $x\in \DD$,}
\beq\label{lem1}
\int_\DD |G_t(x,y)|^2dy\leq Ct^{-\frac d4}.
\nneq

\item[2.]
{\rm
For $\gamma<1-\frac d4$, there exists $C>0$ such that for any $0\le t_0<t\le T$ and $x\in\DD$,
\beq\label{lem2}
\int_{t_0}^t\int_\DD|G_{t-s}(x,y)|^2dyds\leq C|t-t_0|^\gamma.
\nneq
}

\item[3.]
{\rm
Suppose that $\rho \in [1,\infty), p\in [\rho,\infty), \beta\ge 1$ and $\kappa=\frac 1p-\frac {1}{\rho}+1 \in [0,1]$
(if $d=3$, we also need $\frac 1\kappa<3$; and if $d=2, \frac 1\kappa\neq \infty$).
For $v\in L^\beta([0,T],L^{\rho}(\DD)), 0\le t_0<t\le T$ and $x \in\DD$, define
$$
J(v)(t_0,t,x):=\int_{t_0}^t\int_\DD\Delta G_{t-s}(x,y)v(s,y)dyds.
$$
Then $J$ is a bounded operator from $L^\beta([0,T],L^{\rho}(\DD))$ into $L^\infty([0,T],L^p(\DD))$ and there exists a constant $C>0$ such that the following inequalities hold:
\begin{enumerate}

\item
for all $\beta > \frac{1}{\frac12+\frac d4(\kappa-1)}$,
\beq\label{lem32}
\|J(v)(t_0,t,\cdot)\|_p \leq Ct^{\frac12+\frac d4(\kappa-1)-\frac{1}{\beta}}\|v(\ast,\cdot)\|_{L^\beta([t_0,t],L^{\rho}(\DD))},
\nneq
with $0\le t_0 < t \le T$.
\item
for any  $\gamma \in \left(0,\frac12+\frac d4(\kappa-1)\right)$, $\beta > \frac{1}{\frac12+\frac d4(\kappa-1)-\gamma}$,
\beq\label{lem33}
\left\|J(v)(0,t',\cdot)-J(v)(0,t,\cdot)\right\|_p \le C|t'-t|^\gamma\|v(\ast,\cdot)\|_{L^\beta([0,T],L^{\rho}(\DD))},
\nneq
with $0\le t< t'\le T.$

\end{enumerate}
}

\end{itemize}
\end{lemma}

Next we give an estimate about  the convergence of $u^\e$ as $\e\to 0$.

For any $q\geq p$, $t\in[0,T]$ and $M>0$, let \beq\label{Atm}\AA^\e(t,M):=\left\{ \omega\in\Omega;\ \sup_{s\in[0,t]}\left\|u^\e(s,\cdot)\right\|_p^q\leq M\right\}.\nneq

\begin{proposition}\label{Prop 1}
{\rm
Under ${\bf(H)}$,   for any $q\geq p$,
there exists some positive constant $C$ such that for any $M>0$,
\begin{equation}\label{eq prop 1}
\sup_{t\in[0,T]}\ee\left[I_{\AA^\e(t,M)}\left\|u^\e(t,\cdot)-u^0(t,\cdot)\right\|_p^q\right]\leq C\e^{\frac{q}{2}}\rightarrow 0, \quad \text{as} \ \e\rightarrow 0.
\end{equation}
}
\end{proposition}

\begin{proof} Since for any $ t\in[0,T]$,
\begin{align*}
 u^{\e}(t,x)-u^0(t,x)&=\sqrt{\e}\int_0^t\int_{\DD} G_{t-s}(x,y)\sigma (u^\e(s,y))W(ds,dy)\notag\\
                   &\ \ \ +\int_0^t\int_{\DD}\Delta G_{t-s}(x,y)\left[f(u^\e(s,y))-f(u^0(s,y))\right]dyds\notag\\
                   &=:T^\e_1(t,x)+T^\e_2(t,x),
\end{align*}
we obtain that for any $q\geq p$,
\begin{equation}\label{T}
\left\|u^\e(t,\cdot)-u^0(t,\cdot)\right\|_p^q\leq C\left(\left\|T^\e_1(t,\cdot)\right\|_p^q +\left\|T^\e_2(t,\cdot)\right\|_p^q\right),
\end{equation}
for some positive constant $C$.

For $T^\e_1(t,x)$, firstly, it follows from H\"older's inequality that
$$
\ee\left[\left\|T_1^\e(t,\cdot)\right\|_p^q\right]=\ee\left(\int_{\DD}\left|T_1^\e(t,x)\right|^p dx\right)^{\frac{q}{p}}
\le C\int_{\DD}\ee\left[\left|T_1^\e(t,x)\right|^q\right] dx.
$$

Then BDG inequality, the boundedness of $\sigma$ and Eq.\eqref{lem2} imply that
\begin{align}
\ee\left[\left|T_1^\e(t,x)\right|^q\right]\notag
\leq &C \e^{\frac{q}{2}}\ee\left[\left|\int_0^t\int_{\DD} G^2_{t-s}(x,y)\sigma^2\left(u^\e(s,y)\right)dyds\right|^{\frac{q}{2}}\right]\notag \\
\leq & C\e^{\frac{q}{2}}\left|\int_0^t\int_{\DD} G^2_{t-s}(x,y)dyds\right|^{\frac{q}{2}}\notag\\
\leq &C(T)\e^{\frac{q}{2}}.
\end{align}
Therefor, we have
\beq\label{T1}
\ee\left[\left\|T_1^\e(t,\cdot)\right\|_p^q\right]\leq C(T)\e^{\frac{q}{2}}.
\nneq

By Eq.\eqref{lem32} in Lemma \ref{lemma1}, for any $\rho\in [1,p]$, $\kappa=1/p-1/\rho+1 \in [0,1]$ and $\beta \in\left(\frac{1}{\frac12+\frac{d}{4}(\kappa-1)},q\right]$, we obtain that
\begin{align}\label{T20}
&\ee \left[I_{\AA^\e(t,M)}\left(\int_{\DD}|T_2^\e(t,x)|^p dx\right)^{\frac{q}{p}}\right]\notag\\
\leq& C(T)\ee\left[I_{\AA^\e(t,M)}\left(\int_0^t\left\|f(u^\e(s,\cdot))- f(u^0(s,\cdot))\right\|^\beta_\rho ds\right)^{\frac{q}{\beta}}\right] \notag\\
\leq& C(T)\ee\left[I_{\AA^\e(t,M)}\int_0^t\left\|f(u^\e(s,\cdot))- f(u^0(s,\cdot))\right\|_\rho^qds\right] \notag\\
\leq& C(T)\ee\left[\int_0^t I_{\AA^\e(s,M)} \left\|f(u^\e(s,\cdot))-f(u^0(s,\cdot))\right\|_\rho^q ds\right],
\end{align}
where the last inequality is by  the fact that $\AA^\e(t,M)\subset\AA^\e(s,M)$ for any $0\le s\le t$.

Since $f$ is a polynomial function of degree $3$,  taking $\rho =\frac {p}{3}$, by  H\"older's inequality,  we have
\begin{equation}\label{f}
\left\|f(u^\e(s,\cdot))-f(u^0(s,\cdot))\right\|_\rho^q\leq C\left\|u^\e(s,\cdot)-u^0(s,\cdot)\right\|_p^q\left(1+\left\|u^\e(s,\cdot)\right\|_p^{2q}+\left\|u^0(s,\cdot)\right\|_p^{2q}\right).
\end{equation}
Then the definition of $\AA^\e(t,M)$  and estimation \eqref{eq u 0 estimate} yield that
\beq\label{T2}
\ee\left[I_{\AA^\e(t,M)}\left\|T_2^\e(t,\cdot)\right\|_p^q\right]\leq C(T,M)\int_0^t\ee\left[I_{\AA^\e(s,M)}\left\|u^\e(s,\cdot)-u^0(s,\cdot)\right\|_p^q\right]ds.
\nneq
Combining (\ref{T1}) with (\ref{T2}), we have
$$
\ee\left[I_{\AA^\e(t,M)}\left\|u^\e(t,\cdot)-u^0(t,\cdot)\right\|_p^q\right]
\leq C(T) \e^{\frac{q}{2}}+C(T,M)\int_0^t\ee\left[I_{\AA^\e(s,M)}\left\|u^\e(s,\cdot)-u^0(s,\cdot)\right\|_p^q\right]ds.
$$
By Gronwall's inequality, we  obtain the desired  estimate.

The proof is complete.
\end{proof}

\subsection{Proof of Theorem \ref{CLT}}
\begin{proof}[Proof of Theorem \ref{CLT}]
Denote $Y^\e:=(u^\e-u^0)/\sqrt{\e}$.  We will prove that
\begin{equation}\label{CLT pro}
\left\|Y^\e-Y\right\|_{\alpha,p}^q \longrightarrow 0\ \text{ in probability as } \e\longrightarrow 0,
\end{equation}
with $\alpha\leq \frac{\gamma}{4}$ and $\alpha < \frac{1}{2}(1-\frac{d}{4})$ and  $q\geq p$.

By \eqref{eq u e estimate}, we have $P(\AA^\e(T, M)^c)\to 0 $ as $\e\to0$ and $M\to \infty$, thus we only need to show that

\begin{equation}\label{CLT p}
\lim_{\e\to 0}\left\|Y^\e-Y\right\|_{\alpha,p}^q=0 \ \text {  in probability on }\AA^\e(T, M).
\end{equation}
The task of the following part  is to prove that
\begin{equation}\label{CLT exp}
\lim_{\e\to 0}\ee\left[I_{\AA^\e(T, M)}\left\|Y^\e-Y\right\|_{\alpha,p}^q\right]=0,
\end{equation}
 which is stronger than \eqref{CLT p}.
Denote $$V^\e:=Y^\e-Y.$$
By Lemma A.1 in \cite{CWM},  to prove Eq.\eqref{CLT exp}, we only need to verify the following two conditions for $V^\e$:

\begin{itemize}
  \item[(A1).] for all $t\in[0,T]$,
  $$\lim_{\e\rightarrow 0}\ee\left[I_{\AA^\e(t, M)}\left\|V^{\e}(t,\cdot)\right\|_p^q\right]=0;$$
  \item[(A2).] there exists $\gamma>0$,  such that for all $t',t \in[0,T]$,
$$
\ee\left[I_{\AA^\e(T, M)}\left\|V^\e(t',\cdot)-V^\e(t,\cdot)\right\|^q_p\right]\leq C|t'-t|^{\gamma q},
$$
where $C$ is a positive constant independent of $\e$.
\end{itemize}


\noindent{\bf Step 1.} First, we prove (A1) for $V^\e(t,x)$.
Notice that
\begin{align*}
 &Y^\e(t,x)-Y(t,x)\notag\\
 =&\int_0^t\int_{\DD}G_{t-s}(x,y)\left[\sigma(u^\e(s,y))-\sigma(u^0(s,y))\right]W(ds,dy)\notag\\
  &+\int_0^t\int_{\DD}\Delta G_{t-s}(x,y)\left(\frac{f(u^\e(s,y))-f(u^0(s,y))}{\sqrt \e}-f'(u^0(s,y))Y^\e(s,y)\right)dyds\notag\\
  &+\int_0^t\int_{\DD}\Delta G_{t-s}(x,y)f'(u^0(s,y))(Y^\e(s,y)-Y(s,y))dyds\notag\\
 =:&I_1^\e(t,x)+I_2^\e(t,x)+I_3^\e(t,x),
 \end{align*}
thus
\begin{align}\label{I}
\ee\left[I_{\AA^\e(t, M)}\left\|V^{\e}(t,\cdot)\right\|_p^q\right]\leq C\sum_{i=1}^3\ee\left[I_{\AA^\e(t, M)}\left\|I_i^{\e}(t,\cdot)\right\|_p^q\right].
\end{align}

 By H\"{o}lder's inequality, BDG's inequality and the Lipschitz continuity of $\sigma$,  we  can deduce that
\begin{align}\label{I11}
 &\ee\left[I_{\AA^\e(t,M)}\|I_1^\e(t,\cdot)\|^q_p\right]\notag\\
\le& C \ee\left[I_{\AA^\e(t,M)}\int_\DD |I_1^\e(t,\cdot)|^qdx\right]\notag\\
=& C \ee\left[I_{\AA^\e(t,M)}\int_\DD \left|\int_0^t\int_\DD G_{t-s}(x,y)\left[\sigma(u^\e(s,y))-\sigma(u^0(s,y))\right]W(ds,dy) \right|^q dx\right]\notag\\
\le& C\int_\DD\ee\left[ I_{\AA^\e(t,M)}\left(\int_0^t\int_\DD G^2_{t-s}(x,y)\left|u^\e(s,y)-u^0(s,y)\right|^2 dyds\right)^{\frac q2}\right]dx.
\end{align}
By H\"{o}lder's inequality, for some $0<r<1$, the above inequality implies that
\begin{align}\label{I12}
 &\ee\left[I_{\AA^\e(t,M)}\|I_1^\e(t,\cdot)\|^q_p\right]\notag\\
\leq& C\int_\DD \ee\left[ I_{\AA^\e(t,M)}\left(\int_0^t\int_\DD G^2_{t-s}(x,y)dyds\right)^{\frac {qr}{2}}\left(\int_0^t\int_\DD G^2_{t-s}(x,y)\left|u^\e(s,y)-u^0(s,y)\right|^{\frac{2}{1-r}}dyds \right)^{\frac {q(1-r)}{2}}\right]dx\notag\\
\le&  C(T)\int_\DD \ee \left[I_{\AA^\e(t,M)}\left(\int_0^t\int_\DD G^2_{t-s}(x,y)\left|u^\e(s,y)-u^0(s,y)\right|^{\frac{2}{1-r}}dyds \right)^{\frac {q(1-r)}{2}}\right]dx,
\end{align}
where  Eq.\eqref{lem2} is used in the last inequality.

Choosing $1-r=\frac 2q$ for some $q>2$, by Eq.\eqref{lem1} in Lemma \ref{lemma1} and H\"{o}lder's inequality,
\begin{align}\label{I13}
 &\ee\left[I_{\AA^\e(t,M)}\|I_1^\e(t,\cdot)\|^q_p\right]\notag\\
\le&C(T)\ee\left[ I_{\AA^\e(t,M)}\int_0^t\int_\DD \left(\int_\DD G^2_{t-s}(x,y)dx\right) \left|u^\e(s,y)-u^0(s,y)\right|^q dyds\right]\notag\\
\le&C(T)\ee\left[ I_{\AA^\e(t,M)}\int_0^t (t-s)^{-\frac d4} \left\|u^\e(s,\cdot)-u^0(s,\cdot)\right\|_q^q ds\right]\notag\\
\le&C(T)\ee\left[ \left(\int_0^t (t-s)^{-\frac {5d}{16}}ds\right)^{\frac 45} \left(\int_0^tI_{\AA^\e(s,M)}\left\|u^\e(s,\cdot)-u^0(s,\cdot)\right\|_q^{5q} ds\right)^{\frac 15}\right]\notag\\
\le&C(T)\ee\left[\left(\int_0^t I_{\AA^\e(s,M)}\left\|u^\e(s,\cdot)-u^0(s,\cdot)\right\|_q^{5q}ds\right)^{\frac 15}\right].
\end{align}
Taking $p=q$, Proposition \ref{Prop 1} yields that
\begin{equation}\label{I1}
 \ee\left[I_{\AA^\e(t,M)}\|I_1^\e(t,\cdot)\|^q_p\right]\le C(T)\e^{\frac q 2}.
\end{equation}

Notice that $u^{\e}=u^0+\sqrt\e Y^{\e}$. By the mean theorem for derivatives,  there exists a random field $\xi^\e(t,x)$ taking values in $(0,1)$ such that
\begin{align*}
\frac{f(u^\e)-f(u^0)}{\sqrt\e}=f'\big(u^0+\sqrt\e \xi^\e Y^{\e} \big)Y^\e.
\end{align*}
By  (H.2), we have
\begin{align}\label{eq b1}
\left|\frac{f(u^\e)-f(u^0)}{\sqrt\e}-f'(u^0)Y^{\e}\right|
&=\left|\left[f'(u^0+\sqrt\e \xi^\e Y^{\e} )-f'(u^0)\right] Y^{\e}\right|\notag\\
&\le C\sqrt{\e}\left(2|u^0|+\sqrt\e|Y^\e|+1\right)|Y^{\e}|^2.
\end{align}
By the same calculation as that in  Eq.\eqref{T20}, we have for any $1\le\rho\le p$
\begin{align}
&\ee\left[I_{\AA^\e(t,M)}\|I_2^\e(t,\cdot)\|^q_p\right]\notag\\
\le& C(T)\ee\left[\int_0^tI_{\AA^\e(s,M)}\left\|\frac{f(u^\e(s,\cdot))-f(u^0(s,\cdot))}{\sqrt \e}-f'(u^0(s,\cdot))\right\|^q_\rho ds\right]\notag\\
\leq& \e^{\frac {q}{2}}C(T)\ee\left[\int_0^tI_{\AA^\e(s,M)}\left\|\left(2|u^0(s,\cdot)|+\sqrt\e|Y^\e(s,\cdot)|+1\right)|Y^{\e}(s,\cdot)|^2\right\|^q_\rho ds\right].
\end{align}
Therefor, H\"{o}lder's inequality, estimation \eqref{eq u 0 estimate} and Proposition \ref{Prop 1} jointly yield that
\begin{equation}\label{I2}
\ee\left[I_{\AA^\e(t,M)}\left\|I_2^\e(t,\cdot)\right\|^q_p\right]\le \e^{\frac {q}{2}} C(T,M).
\end{equation}


Similarly, we have for any $1\le\rho\le p$
\begin{align}\label{I3}
\ee\left[I_{\AA^\e(t,M)}\left\|I_3^\e(t,\cdot)\right\|^q_p\right]
&\leq C(T)\ee \left[ \int_0^t I_{\AA^\e(s,M)}\left\|f'(u^0(s,\cdot))V^\e(s,\cdot)\right\|_\rho^qds\right]\notag\\
&\leq C(T)\int_0^t\ee\left[I_{\AA^\e(s,M)}\left\|V^\e(s,\cdot)\right\|_p^q \right]ds.
\end{align}
Putting \eqref{I}, \eqref{I1}, \eqref{I2} and \eqref{I3} together, we have
$$
\ee\left[I_{\AA^\e(t, M)}\left\|V^{\e}(t,\cdot)\right\|_p^q\right]
\le C(T,M)\left(\e^{\frac{q}{2}}+ \int_0^t\ee\left[I_{\AA^\e(s, M)}\left\|V^{\e}(s,\cdot)\right\|_p^q\right]ds\right).
$$
By Gronwall's inequality, we have
\begin{equation}\label{bound V}
\ee\left[I_{\AA^\e(t, M)}\left\|V^{\e}(t,\cdot)\right\|_p^q\right]\le C(T,M)\e^{\frac{q}{2}}\rightarrow 0,\quad \text{as}\ \e\rightarrow 0,
\end{equation}
which implies  (A1).\\

\noindent{\bf Step 2.}
 Notice that for any   $0\le t<t'\le T$ and   $x\in\DD$,
 $$
  V^\e(t',x)- V^\e(t,x)= \left(Y^\e(t',x)-Y^\e(t,x)\right)-\left(Y(t',x)-Y(t',x)\right).
 $$
 To verify (A2) for $V^\e(t',\cdot)- V^\e(t,\cdot)$, it is sufficient to prove that for $Y^\e(t',\cdot)- Y^\e(t,\cdot)$ and $Y(t',\cdot)- Y(t,\cdot)$, respectively.  Since the H\"older regularity of $Y(t',\cdot)- Y(t,\cdot)$ is given in Eq.\eqref{holder-Y}, we only need to  give the proof for $Y^\e(t',\cdot)- Y^\e(t,\cdot)$,  that is  there exists a constant $\gamma>0$ such that, for all $t,t'\in[0,T]$,
 \beq\label{Ye holder}
 \ee\left[I_{\AA^\e(t, M)}\left\|Y^\e(t',\cdot)-Y^\e(t,\cdot)\right\|_p^q\right]\le C|t'-t|^{\gamma q}.
 \nneq
 Since for any $t\in[0,T]$, $\frac{t}{T}\in[0,1]$, we only need to prove the existence of $\gamma$ for any $t,t'\in[0,1]$.

 Notice that
\begin{align}\label{eq Y e}
Y^{\e}(t,x)&=\int_0^t\int_{\DD}G_{t-s}(x,y)\sigma(u^\e(s,y))W(ds,dy)\notag\\
           &\quad +\int_0^t\int_{\DD}\Delta G_{t-s}(x,y)\frac{f(u^\e(s,y))-f(u^0(s,y))}{\sqrt \e}dyds\notag\\
           &=:J_1^\e(t,x)+J_2^\e(t,x).
\end{align}
By the proof of Theorem 1.4 in \cite{CW}, there exists a constant $\gamma_1\in\left(0, \frac 12(1-\frac d4)\right)$ such that
\beq\label{J1}
\ee\left[\left\|J_1^\e(t',\cdot)-J_1^\e(t,\cdot)\right\|_p^q\right]\le  C|t'-t|^{\gamma_1 q}.
\nneq
For the second term,
by Eq.\eqref{lem33}  and H\"{o}lder inequality, for any $\rho\in [1,p]$, $\kappa=1/q-1/p+1 \in [0,1]$, there exist $\gamma_2 \in\left(0, \frac 12+\frac d4(\kappa-1)\right)$ and $\beta\in \left(\frac{1}{\frac 12+\frac d4(\kappa-1)-\gamma_2}, q\right)$,  such that
\begin{align}\label{J2}
&\ee\left[I_{\AA^\e(t, M)}\left\|J_{2}^\e(t',\cdot)-J_{2}^\e(t, \cdot)\right\|_p^q\right]\notag\\
\le & C |t'-t|^{\gamma_2 q} \ee\left[I_{\AA^\e(t, M)}\left(\int_0^t\left\|\frac{f(u^\e(s,\cdot))-f(u^0(s,\cdot))}{\sqrt \e}\right\|_\rho^\beta ds\right)^{q/\beta}\right]\notag\\
\le & C |t'-t|^{\gamma_2 q} \ee\left[\int_0^t I_{\AA^\e(s, M)}\left\|\frac{f(u^\e(s,\cdot))-f(u^0(s,\cdot))}{\sqrt \e}\right\|_\rho^q ds\right]\notag\\
\le & C(T, M) |t'-t|^{\gamma_2 q},
\end{align}
where the last inequality is due to  the same calculation as that of Eq.\eqref{f}.

Combining \eqref{eq Y e}-\eqref{J2},  choosing $\gamma=\gamma_1\wedge\gamma_2$, we get \eqref{Ye holder}.

The proof is complete.
\end{proof}

\section{Moderate Deviation Principle}


For any   $N \in \nn$, let
$$
\HH_N :=\left\{ v\in L^2([0,T]\times \DD);\ \|v\|:=\int_0^T\int_{\DD}v^2(t,x)dxdt\le N\right\},
$$
which is a compact metric space of $L^2([0,T]\times \DD)$ in the weak topology.
Denote $\GG$  the set of predictable processes belonging to $ L^2([0,T]\times \DD)$ a.s., and
$$
\GG_N :=\{ v\in \GG;\ v(\omega)\in \HH_N,\ \pp\text{-a.s.}\}.
$$
For any $\e\in(0,1]$ and $v\in\GG_N$, consider the controlled equation $Z^{\e,v}$ defined by
\begin{align}\label{eq Z e}
&Z^{\e,v}(t,x)\notag\\
 =&\frac{1}{h(\e)}\int_0^t\int_{\DD}G_{t-s}(x,y)\sigma(u^0(s,y)+\sqrt\e h(\e)Z^{\e,v}(s,y))W(ds,dy)\notag\\
 &  +\int_0^t\int_{\DD}G_{t-s}(x,y)\sigma(u^0(s,y)+\sqrt\e h(\e)Z^{\e,v}(s,y))v(s,y)dyds\notag\\
 &  +\int_0^t\int_{\DD}\Delta G_{t-s}(x,y)\frac{f(u^0(s,y)+\sqrt{\e}h(\e)Z^{\e,v}(s,y))-f(u^0(s,y))}{\sqrt\e h(\e)}dyds.
\end{align}
Following  the proof of Theorem 3.2 in \cite{BM},   one can prove that Eq.\eqref{eq Z e} admits a unique solution satisfying \beq\label{eq z e estimate}
\sup_{\e\le 1}\sup_{v\in\GG_N}\sup_{t\in[0,T]}\ee\left[\left\|Z^{\e,v}(t,\cdot)\right\|_p^q\right]<+\infty,
\nneq
for $q\geq p$.

\vskip0.3cm

Consider the following conditions which correspond to the weak convergence approach in our setting:

\begin{itemize}
 \item[({\bf a})]For any family $\{v^\e;\ \e>0\}\subset\GG_N$ which converges in distribution as $\HH_N$-valued random elements to $v\in\GG_N$ as $\e\rightarrow 0$,
       $$
       \lim_{\e\rightarrow 0}Z^{\e,v^{\e}}=Z^v\quad \text{ in distribution},
       $$
       as $\EE_{\alpha}$-valued random variables, where $Z^v$ denotes the solution of Eq.\eqref{eq sk} corresponding to the $\HH_N$-valued random variable $v$ (instead of a deterministic function);
 \item[({\bf b})]The set $\{Z^v;\ v\in \HH_N\}$ is a compact set of $\EE_{\alpha}$, where $Z^v$ is the solution of Eq.\eqref{eq sk}.
\end{itemize}

\vskip0.3cm

\bprf[The proof of Theorem \ref{MDP}]

According to \cite[Theorem 6]{BDM}, we need to prove that Condition  ({\bf a}), ({\bf b}) are fulfilled. Firstly, we verify Condition ({\bf a}).

By the Skorokhod representation theorem, there exist a probability space $(\bar\Omega,\bar\FF,(\bar\FF_t)_{t\geq 0},\bar\pp)$, and, on this basis, a sequence of independent Brownian motions $\bar W$ and also a family of $\bar \FF_t$-predictable processes $\{\bar v^\e;\ \e>0\}$, $\bar v$  taking values in $\HH_N$, $\bar\pp$-a.s., such that the joint law of $(v^\e,v, W)$ under
$\pp$ coincides with that of  $(\bar v^\e,\bar v, \bar W)$ under $\bar\pp$ and
\beq\label{eq conver weak}
\lim_{\e\rightarrow0}\int_0^T\int_{\DD}(\bar v^\e(s,y)-\bar v(s,y))g(s,y)dyds=0, \quad\forall g\in L^2([0,T]\times\DD),\ \bar \pp\mbox{-a.s.}.  \ \
\nneq
Let $\bar Z^{\e,\bar v^\e}$ be the solution of a similar equation to \eqref{eq Z e} replacing $v$  by $\bar v^\e$ and $W$ by $\bar W$, and
let $\bar Z^{\bar v}$ be the solution of Eq.\eqref{eq sk} corresponding to the $\HH_N$-valued random variable $\bar v$.

Now, we shall prove that for any $q\geq p$
and  $\alpha\leq \frac{\gamma}{4}, \alpha < \frac{1}{2}(1-\frac{d}{4})$,
\beq\label{eq b conv}
\lim_{\e\rightarrow0}\bar \ee\left[ \left\|  \bar Z^{ \e, \bar v^\e}- \bar Z^{\bar v} \right\|_{\alpha,p}^q\right]=0,
\nneq
which implies the validity of Condition ({\bf a}). Here the expectation in \eqref{eq b conv} refers to the probability $\bar \pp$.

From now on, we drop the bars in the notation for the sake of simplicity, and we denote
$$
X^{\e,v^\e,v}:=Z^{\e,v^\e}-Z^v.
$$

By Lemma A1 in \cite{CM}, in order to prove \eqref{eq b conv}, it is sufficient to prove that:

\begin{itemize}
  \item[(1)]For all $t\in[0,T]$,
   \beq\label{c1}
   \lim_{\e\to 0}\ee\left[\left\|X^{\e,v^\e,v}(t,\cdot)\right\|_p^q\right] =0.
   \nneq
  \vskip0.2cm

  \item[(2)]There exists a positive constant $C$ and $\gamma>0$ such that: for all $t,t'\in[0,t]$
  \beq\label{c2}
  \sup_{\e\in(0,1]}\ee \left[\left\|X^{\e,v^\e,v} (t',\cdot)-X^{\e,v^\e,v} (t,\cdot)\right\|^q_p\right]\le C|t-t'|^{\gamma q}.
  \nneq
\end{itemize}

According  to  the proof of Theorem 4.1 in \cite{BM}, one can get the H\"{o}lder regularities for   $Z^v(t,\cdot), Z^{\e,v^{\e}}(t,\cdot)$,  that is for some  $\gamma>0 $,
$$
\ee\left[\left\|Z^v (t',\cdot)-Z^{v} (t,\cdot)\right\|^q_p\right]\le C|t-t'|^{\gamma q},
$$
and
$$
  \sup_{\e\in(0,1]}\ee \left[\left\|Z^{\e,v^\e} (t',\cdot)-Z^{\e,v^\e} (t,\cdot)\right\|^q_p\right]\le C|t-t'|^{\gamma q}.
  $$
Thus we get \eqref{c2}. Next we prove \eqref{c1}.

 Notice that for any $(t,x)\in[0,T]\times \DD$,
\begin{align}\label{eq A0}
&Z^{\e,v^\e}(t,x)-Z^{v}(t,x)\notag\\
=&\frac{1}{h(\e)}\int_0^t\int_{\DD} G_{t-s}(x,y)\sigma(u^0(s,y)+\sqrt\e h(\e)Z^{\e,v^\e}(s,y))W(ds,dy)\notag\\
&+\Bigg\{\int_0^t\int_{\DD}G_{t-s}(x,y)\sigma(u^0(s,y)+\sqrt\e h(\e)Z^{\e,v^\e}(s,y))v^\e(s,y)dyds\notag\\
&\ \ \ \ \ \  -\int_0^t\int_{\DD}G_{t-s}(x,y)\sigma(u^0(s,y))v(s,y)dyds\notag\Bigg\}\\
&+\Bigg\{\int_0^t\int_{\DD}\Delta G_{t-s}(x,y)\frac{f(u^0(s,y)+\sqrt\e h(\e)Z^{\e,v^\e}(s,y)) -f(u^0(s,y))}{\sqrt\e h(\e)}dyds\notag\\
&\ \ \ \ \ \  -\int_0^t\int_{\DD}\Delta G_{t-s}(x,y)f'(u^0(s,y))Z^v(s,y)dyds\Bigg\}\notag\\
=:&A_1^\e(t,x)+A_2^\e(t,x)+A_3^\e(t,x).
\end{align}
{\bf Step 1}.  For the first term $A_1^\e(t,x)$, following the same calculation  as that for $\ee\left[\left\|T_1^{\e}(t,\cdot)\right\|^q_p\right]$ in the proof of Proposition \ref{Prop 1}, we have that
\begin{equation}\label{A1}
\ee\left[\left\|A_1^{\e}(t,\cdot)\right\|^q_p\right]\le \frac{C(T)}{h(\e)^q},
\end{equation}
for some positive constant $C(T)$.

{\bf Step 2}.  For the  second term, we  divide it  into two terms:
\begin{align}\label{A2}
&A_2^{\e}(t,x)\notag\\
=& \int_0^t\int_{\DD}G_{t-s}(x,y)\left[\sigma(u^0(s,y)+\sqrt\e h(\e)Z^{\e,v^\e}(s,y))-\sigma(u^0(s,y))\right]v^\e(s,y)dyds\notag\\
&  +\int_0^t\int_{\DD}G_{t-s}(x,y)\sigma(u^0(s,y))[v^\e(s,y)-v(s,y)]dsdy\notag\\
=:&A_{2,1}^\e(t,x)+A_{2,2}^\e(t,x).
\end{align}
By Cauchy-Schwarz inequality and the Lipschitz continuity of $\sigma$, we obtain that
\begin{align}\label{A211}
\ee\left[\left\|A_{2,1}^\e(t,\cdot)\right\|^q_p\right]
\leq &C(\sqrt\e h(\e))^q\ee\left[\|v^\e\|^{\frac q2} \cdot\left\|\left(\int_0^t\int_{\DD} G^2_{t-s}(\cdot,y)\left|Z^{\e,v^\e}(s,y)\right|^2dyds\right)^{\frac 12}\right\|_p^q\right]\notag\\
\leq &CN^{\frac q2}(\sqrt\e h(\e))^q\ee\left[\left\|\left(\int_0^t\int_{\DD} G^2_{t-s}(\cdot,y)\left|Z^{\e,v^\e}(s,y)\right|^2dyds\right)^{\frac 12}\right\|_p^q\right].
\end{align}

 H\"older's inequality, Eq.\eqref{lem1} and Eq.\eqref{lem2} imply that
\begin{align}\label{A212}
&\ee\left[\left\|\left(\int_0^t\int_{\DD} G^2_{t-s}(\cdot,y)\left|Z^{\e,v^\e}(s,y)\right|^2dyds\right)^{\frac 12}\right\|_p^q\right]\notag\\
 \leq &\ee\left[\int_{\DD}\bigg(\int_0^t\int_{\DD} G^2_{t-s}(x,y)dyds\bigg)^{\frac{p}{2} -1}\cdot\left(\int_0^t\int_{\DD} G^2_{t-s}(x,y)|Z^{\e,v^\e}(s,y)|^p dyds\right)dx\right]^{\frac qp}\notag\\
\leq&C(T)\ee \left[\int_0^t\left(\int_{\DD} G^2_{t-s}(x,y)dx\right) \cdot \left(\int_{\DD}|Z^{\e,v^\e}(s,y)|^p dy\right)ds\right]^{\frac qp}\notag\\
\leq&C(T)\ee\left(\int_0^t(t-s)^{-\frac d4}\left\|Z^{\e,v^\e}(s,\cdot)\right\|^p_p ds\right)^{\frac qp}\notag\\
\leq&C(T)\ee\left(\int_0^t(t-s)^{-\frac d4}\left\|Z^{\e,v^\e}(s,\cdot)\right\|^q_p ds\right).
\end{align}
Thus by H\"older's inequality and estimation  (\ref{eq z e estimate}), we get
\begin{equation}\label{A21}
\ee\left[\left\|A_{2,1}^\e(t,\cdot)\right\|^q_p\right] \leq C(T,N)(\sqrt\e h(\e))^q,
\end{equation}

Next,  we will show that
\beq\label{A22}
 \lim_{\e \rightarrow 0}\ee\left[\left\|A^{\e}_{2,2}(t,\cdot)\right\|^q_p\right]=0.
\nneq
For any $(t,x)\in[0,T]\times\DD$, by the boundedness of $\sigma$ and Eq.\eqref{lem2}, we have
$$
\int_0^T\int_\DD G^2_{t-s}(x,y)\sigma^2(u^0(s,y))dyds\leq C\int_0^T\int_\DD G_{t-s}^2(x,y)dyds< +\infty,
$$
which implies that for any $(t,x)\in[0,T]\times\DD$, the function $\big\{G_{t-s}(x,y)\sigma(u^0(s,y));\ (s,y)\in[0,T]\times\DD \big\}$ takes its values in $\HH_N$ for some $N\in\nn$. Since $v^\e\to v$ weakly in $\HH_N$, we obtain that:

   \beq\label{eq conv}
   \lim_{\e\rightarrow0}A^{\e}_{2,2}(t,x)=0,\  \ \text{a.s.}.
   \nneq
By Cauchy-Schwarz's inequality on  $\HH_N$ and the facts that $v^{\e},v\in \HH_N$ a.s., we  obtain that
\begin{align}\label{eq bound}
\left|A^{\e}_{2,2}(t,x)\right|
\leq &\left(\int_0^t\int_\DD G_{t-s}^2(x,y)\sigma^2\big(u^0(s,\cdot)\big)dyds\right)^{\frac12}\cdot\left(\int_0^t\int_\DD(v^{\e}(s,y)-v(s,y))^2dyds\right)^{\frac12}\notag\\
\leq &C(N)\left(\int_0^t\int_\DD G_{t-s}^2(x,y)\sigma^2\big(u^0(s,\cdot)\big)dyds\right)^{\frac12}\notag\\
                      \leq & C(N,T),
\end{align}
where $C(N, T)$ is independent of $(\e,t,x)$.  This implies that $A^{\e}_{2,2}(t,x)$ is  uniformly bounded a.s..
By the dominated convergence theorem,  we obtain \eqref{A22}.

{\bf Step 3}. For the third term,  it is also further divided into two terms
\begin{align}\label{A3}
&A_3^{\e}(t,x)\notag\\
= &\int_0^t\int_{\DD}\Delta G_{t-s}(x,y)\Bigg[\frac{f\big(u^0(s,y)+\sqrt{\e}h(\e)Z^{\e,v^{\e}}(s,y)\big)-f(u^0(s,y)
)}{\sqrt{\e}h(\e)}-f'(u^0(s,y))Z^{\e,v^{\e}}(s,y) \Bigg]dyds\notag\\
 &+\int_0^t\int_{\DD}\Delta G_{t-s}(x,y)f'(u^0(s,y))\big(Z^{\e,v^{\e}}(s,y)-Z^v(s,y)\big)
  dyds\notag\\
 =:& A^\e_{3,1}(t,x)+A^\e_{3,2}(t,x).
\end{align}
By the mean theorem for derivatives,  for each $\e \in(0,1]$, there exist  random fields $\xi^\e(t,x)$ and $\eta^\e(t,x)$ taking values in $(0,1)$ such that
\begin{align*}
&\frac{f\left(u^0+\sqrt{\e}h(\e)Z^{\e,v^{\e}}\right)-f(u^0)}{\sqrt{\e}h(\e)}-f'(u^0)Z^{\e,v^{\e}}\notag\\
=&f'\left(u^0+\sqrt\e h(\e) \xi^\e Z^{\e,v^{\e}}\right)Z^{\e,v^{\e}}-f'(u^0)Z^{\e,v^{\e}}\notag\\
=&\sqrt{\e}h(\e) \xi^\e f''(u^0+\sqrt\e h(\e)\eta^\e \xi^\e Z^{\e,v^{\e}})(Z^{\e,v^{\e}})^2.
\end{align*}
Thus by the same calculation as Eq.\eqref{T20}, we have for any $\rho \in[1,p]$,
\begin{align}\label{A31}
&\ee\left[\left\|A_{3,1}^{\e}(t,\cdot)\right\|_p^q\right]\notag\\
\leq& (\sqrt \e h(\e))^q\ee\left[\left\|\int_0^t\int_\DD\left|\Delta G_{t-s}(x,y)f''(u^0(s,y)+\sqrt\e h(\e) \eta^\e \xi^\e Z^{\e,v^{\e}}(s,y))(Z^{\e,v^{\e}}(s,y))^2\right|dyds\right\|_p^q\right]\notag\\
\leq& C(T)(\sqrt\e h(\e))^q\ee\left[\int_0^t\left\|f''\left(u^0(s,\cdot)+\sqrt\e h(\e) \eta^\e(s,\cdot) \xi^\e(s,\cdot) Z^{\e,v^{\e}}(s,\cdot)\right)|Z^{\e,v^{\e}}(s,\cdot)|^2\right\|^q_\rho ds\right]\notag\\
\leq& C(T)(\sqrt{\e}h(\e))^q,
\end{align}
where the last inequality is due to (H.2), H\"{o}lder's inequality, Eq.\eqref{eq u 0 estimate} and Eq.\eqref{eq z e estimate}.

 Similarly to $A_{3,1}^{\e}$,  one can prove that  for any $\rho \in[1,p]$
\begin{align}\label{A32}
\ee\left[\left\|A_{3,2}^{\e}(t,\cdot)\right\|_p^q\right]
\leq&C(T)\ee\left[\int_0^t\left\|f'(u^0(s,\cdot))(Z^{\e,v^{\e}}(s,\cdot)-Z^v(s,\cdot))\right\|^q_\rho ds\right]\notag\\
\leq &C(T)\int_0^t\ee\left[\left\|Z^{\e,v^{\e}}(s,\cdot)-Z^v(s,\cdot)\right\|^q_p\right]ds.
\end{align}

Putting \eqref{eq A0}-\eqref{A21},  \eqref{A3}-\eqref{A32} together, we have
\begin{align*}
\ee\left[\left\|X^{\e,v^{\e},v}(t, \cdot)\right\|^q_p\right]
 \le&C\Bigg(h^{-q}(\e)+(\sqrt\e h(\e))^q+\ee\left[\left\|A_{2,2}^\e(t, \cdot)\right\|_p^q\right]\notag\\
 &\ \ \ \ \ +(\sqrt\e h(\e))^q+\int_0^t\ee\left[\left\|X^{\e,v^{\e},v}(s, \cdot)\right\|_p^q\right]ds \Bigg).
\end{align*}
By   Gronwall's inequality and \eqref{A22}, we obtain that
\begin{equation*}
\ee\left[\left\|X^{\e,v^{\e},v}(t,\cdot)\right \|_p^q\right]
\leq C\left(h^{-q}(\e)+(\sqrt\e h(\e))^q+\ee\left[\left\|A_{2,2}^\e(t, \cdot)\right\|_p^q\right]\right)
\longrightarrow 0, \ \  \text{as } \e\rightarrow 0.
\end{equation*}
Thus \eqref{c1} is verified, which implies the Condition ({\bf a}).

\vskip0.3cm

Replacing  $\frac{1}{h(\e)}\int_0^t\int_{\DD}G_{t-s}(x,y)\sigma(u^0(s,y)+\sqrt\e h(\e)Z^{\e,v}(s,y))W(ds,dy)$ in the proof of the Condition ({\bf a}),   we obtain that  $Z^v$  is a continuous mapping from $v\in\HH_N$ into $\EE_{\alpha}$.  Since $\HH_N$ is compact in the weak topology of $L^2([0,T]\times \DD)$, $\{Z^v;v\in\HH_N\}$ is a compact subset of $\EE_{\alpha}$. That is  the Condition ({\bf b}).

The proof is complete.
\end{proof}

\vskip0.3cm


\end{document}